\documentclass[12pt,a4paper,oneside]{amsart}
\usepackage[a4paper]{geometry}
\usepackage{mathptmx} 
\DeclareMathAlphabet{\mathcal}{OMS}{cmsy}{m}{n} 

\newtheorem{theorem}{Theorem}
\newtheorem{lemma}{Lemma}
\theoremstyle{remark}
\newtheorem{remark}{Remark}

\newcommand{\set}[2]{\ensuremath{\{ #1 \>|\> #2 \}}}

\usepackage{perpage} 

\MakePerPage[2]{footnote}

\begin{document}

\title[Lie algebras extended over commutative algebras without unit]
{Invariants of Lie algebras extended over commutative algebras without unit}
\author{Pasha Zusmanovich}
\address{}
\email{pasha.zusmanovich@gmail.com}
\date{last revised December 26, 2017}
\thanks{arXiv:0901.1395;
J. Nonlin. Math. Phys. \textbf{17} (2010), Suppl. 1 (special issue in memory of F.A. Berezin), 87--102.}

\begin{abstract}
We establish results about the second cohomology with coefficients
in the trivial module, symmetric invariant bilinear forms, and
derivations of a Lie algebra extended over a commutative associative
algebra without unit. These results provide a simple unified
approach to a number of questions treated earlier in completely
separated ways: periodization of semisimple Lie algebras (Anna
Larsson), derivation algebras, with prescribed semisimple part, 
of nilpotent Lie algebras (Benoist), and presentations of affine
Kac-Moody algebras.
\end{abstract}



\maketitle

\section*{Introduction}

In this paper we consider current Lie algebras, i.e., Lie algebras of
the form $L\otimes A$, where $L$ is a Lie algebra, $A$ is a
commutative associative algebra, and the multiplication in $L\otimes
A$ being defined by the formula
\begin{equation*}
[x \otimes a, y\otimes b] = [x,y] \otimes ab
\end{equation*}
for any $x, y\in L$, $a, b\in A$. Note that $A$ can be considered as the
algebra of functions on the spectrum of $A$, and $L\otimes A$ can
therefore be interpreted as the algebra of ``currents'', as
physicists say, on this spectrum.

Berezin and Karpelevich \cite{berezin} were among the first to study certain class 
of representations of current algebras over local finite dimensional 
commutative algebras $A$. In that, somewhat obscure, paper they 
showed that cohomology of such algebras can be reduced to cohomology of $L$.
In what follows we will consider other types of commutative algebras $A$ and the 
description of cohomology in our case is more involved and interesting.
We are primarily interested in the second cohomology of $L\otimes A$ with
trivial coefficients, the space of symmetric invariant bilinear
forms on $L\otimes A$, and the algebras of  derivations of $L\otimes
A$. 

These invariants were determined for numerous particular cases
of current Lie algebras (see, for example, \cite{santharoubane}),
the general formulae for the second homology with trivial
coefficients in terms of invariants of $L$ and $A$ were obtained in
\cite{haddi}, \cite{asterisque} and \cite{neeb}, and similar
formulae for the space of symmetric invariant bilinear forms and
derivation algebras were obtained in \cite{asterisque} and
\cite{low}, respectively.

So why return to these settled questions? In all considerations
until now, the algebra $A$ was supposed to have a unit. However,
there are many interesting examples of current algebras where $A$ is
not unital. For example, in \cite{larsson}, the so-called
periodization of semisimple Lie algebras $\mathfrak g$ was
considered, which is nothing but $\mathfrak g\otimes t\mathbb C[t]$.
It is known that the second homology of any nilpotent Lie algebra
with trivial coefficients has interpretation in terms of
presentation of the algebra, so allowing $A$ to be nilpotent allows
us to obtain presentation of $L\otimes A$ irrespective of the
properties of $L$.

It turns out that elementary arguments similar to those in
\cite{low} allow us to extend the above mentioned results to the
case of non-unital $A$. In particular, concerning the second
cohomology and symmetric invariant bilinear forms, we provide
another proof, considerably shorter than all the previous ones even
in the case of unital $A$.

The contents of this paper are as follows. In \S\S 1--3 we establish
the general formulae for $2$-cocycles, symmetric invariant bilinear
forms, and get partial results about derivations of the current Lie
algebras, respectively. This is followed by applications: in \S 4 we
reprove the result from \cite{larsson} about presentations of
periodizations of the semisimple Lie algebras. In passing, we also
mention how to derive from our results the theorem from
\cite{benoist} about semisimple components of the derivation
algebras of certain current Lie algebras, and the Serre defining
relations between Chevalley generators of the non-twisted affine
Kac-Moody algebra.

In all these cases, the absence of unit in $A$ is essential. All
these proofs are significantly shorter than the original ones, and
reveal various almost trivial, but so far unnoticed or unpublished,
links between different concepts and results. These links are,
perhaps, the main virtue of this paper.

It seems that most if not everything considered here can be extended in a
straightforward way to twisted, Leibniz and super settings, but we
will not venture into this, at least for now.

\section*{Notation and conventions}

All algebras and vector spaces are defined over an arbitrary field $K$
of characteristic different from 2 and 3, unless stated otherwise
(some of the results are valid in characteristic $3$, but we will
not go into this).

In what follows, $L$ denotes a Lie algebra, $A$ is an associative
commutative algebra.

Given an $L$-module $M$, let $B^n(L,M)$, $Z^n(L,M)$, $C^n(L,M)$ and
$H^n(L,M)$ denote the space of $n$th degree coboundaries, cocycles,
cochains, and cohomology of $L$ with coefficients in $M$,
respectively (we will be mainly interested in the particular cases
of degree 2 and the trivial module $K$, or degree 1 and the adjoint
module or its dual). Note that $C^2(L,K)$ is the space of all
skew-symmetric bilinear forms on $L$. The space of all symmetric
bilinear forms on $L$ will be denoted as $S^2(L,K)$.

Let $\mathcal Z(L)$, $[L,L]$ and $Der(L)$ denote the center, the commutant
(the derived algebra), and the Lie algebra of derivations of $L$,
respectively. Similarly, let $Ann(A) = \set{a\in A}{Aa = 0}$ and $AA$
denote the annulator and the square of $A$, respectively, and let
$HC^*(A)$ denotes its cyclic cohomology.

A bilinear form $\varphi:L \times L \to K$ is said to be
\textit{cyclic} if
\begin{equation*}
\varphi([x,y],z) = \varphi([z,x],y)
\end{equation*}
for any $x,y,z\in L$. Note that if $\varphi$ is symmetric, this
condition is equivalent to the \textit{invariance} of the form
$\varphi$:
\begin{equation*}
\varphi([x,y],z) + \varphi(y,[x,z]) = 0 ,
\end{equation*}
while if $\varphi$ is skew-symmetric, the notions of cyclic and
invariant forms differ.
Let $\mathcal B(L)$ denote the space of all symmetric bilinear invariant
(=cyclic) forms on $L$. 

Similarly, a bilinear form $\alpha:A \times
A \to K$ is said to be \textit{cyclic} if
\begin{equation*}
\alpha(ab,c) = \alpha(ca,b) ,
\end{equation*}
and \textit{invariant} if
$$
\alpha(ab,c) = \alpha(a,bc)
$$
for any $a,b,c\in A$. 
If the form $\alpha$ is symmetric, it is cyclic if and only if it is invariant.

\section{The second cohomology}\label{cohomology}

\begin{theorem}\label{h2}
Let $L$ be a Lie algebra, $A$ an associative commutative algebra,
and at least one of $L$ and $A$ be finite-dimensional. Then each
cocycle in $Z^2(L\otimes A,K)$ can be represented as the sum of
decomposable cocycles $\varphi\otimes \alpha$, where $\varphi:
L\times L \to K$ and $\alpha: A\times A\to K$ are of one of the
following $8$ types:
\begin{enumerate}
\item
$\varphi([x,y],z) + \varphi([z,x],y) + \varphi([y,z],x) = 0$ and
$\alpha$ is cyclic,
\item
$\varphi$ is cyclic and $\alpha(ab,c) + \alpha(ca,b) + \alpha(bc,a)
= 0$,
\item $\varphi ([L,L],L) = 0$,
\item $\alpha (AA,A) = 0$,
\end{enumerate}
where each of these $4$ types splits into two subtypes: with
$\varphi$ skew-symmetric and $\alpha$ symmetric, and with $\varphi$
symmetric and $\alpha$ skew-symmetric.
\end{theorem}

\begin{proof}
Each cocycle $\Phi \in Z^2(L\otimes A,K)$, being an element of
\begin{equation*}
End (L\otimes A \otimes L \otimes A, K) \simeq End(L\otimes L, K)
\otimes End(A\otimes A,K),
\end{equation*}
can be written in the form $\Phi = \mathop{\sum}\limits_{i\in I}
\varphi_i \otimes \alpha_i$, where $\varphi_i: L \times L \to K$ and
$\alpha_i:A \times A \to K$ are bilinear maps, and $I$ is a finite set of indices 
(this is the place where the assumption of finite-dimensionality is needed). 
Using this representation, write the cocycle equation for an arbitrary triple
$x\otimes a$, $y\otimes b$, $z\otimes c$, where $x,y,z\in L$,
$a,b,c\in A$:
\begin{equation}\label{cocycle-eq}
\mathop{\sum}\limits_{i\in I} \varphi_i([x,y],z) \otimes
\alpha_i(ab,c) +
              \varphi_i([z,x],y) \otimes \alpha_i(ca,b) +
              \varphi_i([y,z],x) \otimes \alpha_i(bc,a) = 0.
\end{equation}
Symmetrizing this equality with respect to $x,y$, we get:
\begin{equation*}
\mathop{\sum}\limits_{i\in I} \Big( \varphi_i([x,z],y) +
\varphi_i([y,z],x) \Big) \otimes
              \Big( \alpha_i(bc,a) - \alpha_i(ca,b) \Big) = 0.
\end{equation*}
On the other hand, cyclically permuting $x,y,z$ in
(\ref{cocycle-eq}) and summing up the 3 equalities  obtained, we
get:
\begin{equation*}
\mathop{\sum}\limits_{i\in I} \Big( \varphi_i([x,y],z) +
\varphi_i([z,x],y) + \varphi_i([y,z],x) \Big)
              \otimes
              \Big( \alpha_i(ab,c) + \alpha_i(bc,a) + \alpha_i(ca,b) \Big) = 0.
\end{equation*}
Applying Lemma 1.1 from \cite{low} to the last two equalities, we
get a partition of the index set $I = I_1 \cup I_2 \cup I_3 \cup
I_4$ such that
\begin{align*}
&\varphi_i([x,z],y) + \varphi_i([y,z],x) = 0, \> \varphi_i([x,y],z)
+ \varphi_i([z,x],y) + \varphi_i([y,z],x) = 0
&\text{for } i\in I_1\phantom{.} \\
&\varphi_i([x,z],y) + \varphi_i([y,z],x) = 0, \> \alpha_i(ab,c) +
\alpha_i(bc,a) + \alpha_i(ca,b) = 0
&\text{for } i\in I_2\phantom{.} \\
&\varphi_i([x,y],z) + \varphi_i([z,x],y) + \varphi_i([y,z],x) = 0,
\> \alpha_i(bc,a) - \alpha_i(ca,b) = 0
&\text{for } i\in I_3\phantom{.} \\
&\alpha_i(bc,a) - \alpha_i(ca,b) = 0, \> \alpha_i(ab,c) +
\alpha_i(bc,a) + \alpha_i(ca,b) = 0 &\text{for } i\in I_4.
\end{align*}
It is obvious that if the characteristic of $K$ is different from
$3$, then $\varphi_i([L,L],L) = 0$ for $i\in I_1$, and
$\alpha_i(AA,A) = 0$ for $i\in I_4$. It is obvious also that
$\varphi_i \otimes \alpha_i$ satisfies the cocycle equation
(\ref{cocycle-eq}) for each $i \in I_1,I_2,I_3,I_4$.

Now write the condition of skew-symmetry of $\Phi$:
\begin{equation}\label{skew-symm}
\mathop{\sum}\limits_{i\in I} \varphi_i(x,y) \otimes \alpha_i(a,b) +
\varphi_i(y,x) \otimes \alpha_i(b,a) = 0
\end{equation}
and symmetrize it with respect to $x, y$:
\begin{align*}
&\mathop{\sum}\limits_{i\in I} (\varphi_i(x,y) - \varphi_i(y,x))
\otimes (\alpha_i(a,b) - \alpha_i(b,a)) = 0 \\
&\mathop{\sum}\limits_{i\in I} (\varphi_i(x,y) + \varphi_i(y,x))
\otimes (\alpha_i(a,b) + \alpha_i(b,a)) = 0 .
\end{align*}
From the last two equalities, using again Lemma 1.1 from \cite{low},
we see that each set $I_1,I_2,I_3,I_4$ can be split further into two
subsets, one having skew-symmetric $\varphi_i$ and symmetric
$\alpha_i$, and the other one having symmetric $\varphi_i$ and
skew-symmetric $\alpha_i$.
\end{proof}

\begin{remark}
As all our bilinear maps are $K$-valued, the cocycles of the form
$\varphi \otimes \alpha$ are, of course, just products of bilinear
maps $\varphi \alpha$. However, we have retained the symbol
$\otimes$, to make it easier to track dependence on the more general
situation of \cite{low}.
\end{remark}

\begin{remark}\label{indep}
The second subdivision in the statement of Theorem \ref{h2}, which
follows from  equality (\ref{skew-symm}), is merely a manifestation
of the vector space isomorphism
\begin{equation}\label{summands1}
C^2(L\otimes A,K) \simeq \left(S^2(L,K) \otimes C^2(A,K)\right)
\>\oplus\> \left(C^2(L,K) \otimes S^2(A,K)\right).
\end{equation}
Let $d\Omega$ be the $2$-coboundary defined by a given linear map
$\Omega: L\otimes A \to K$. The latter can be written in the form
$\Omega = \mathop{\sum}\limits_{i\in I} \omega_i \otimes \beta_i$
for some linear maps $\omega_i:L\to K$ and $\beta_i:A\to K$. Then
\begin{equation*}
d\Omega (x\otimes a, y\otimes b) = \mathop{\sum}\limits_{i\in I}
\omega_i ([x,y]) \otimes \beta_i(ab),
\end{equation*}
i.e., coboundaries always lie in the direct summand $C^2(L,K) \otimes S^2(A,K)$. 
Consequently, nonzero cocycles from different
direct summands in (\ref{summands1}) can never be cohomologically dependent, and the
cocycles from the direct summand $S^2(L,K) \otimes C^2(A,K)$ are
cohomologically independent if and only if they are linearly
independent.
\end{remark}

One may try to formulate Theorem \ref{h2} as a statement about 
$H^2(L\otimes A,K)$, but in full generality this will lead only to
cumbersome 
complications. In each case of interest, one can easily obtain an
information about the cohomology. For example, assuming $A$ contains
a unit, one immediately sees that cocycles of type (i) necessarily
have $\varphi$ skew-symmetric and $\alpha$ symmetric, cocycles of
type (ii) necessarily have $\varphi$ symmetric and $\alpha$
skew-symmetric, and cocycles of type (iv) vanish. This leads to a
known formula for $H^2(L\otimes A,K)$, where the cocycles of type
(i) contribute to the term $H^2(L,K) \otimes A^*$, the cocycles of
type (ii) contribute to the term $\mathcal B(L) \otimes HC^1(A)$, and the
cocycles of type (iii) are non-essential (in terminology of
\cite{asterisque}).

Another, more concrete, application is given in \S \ref{period}.

\section{Symmetric invariant bilinear forms}

\begin{theorem}\label{forms}
Let $L$ be a Lie algebra, $A$ an associative commutative algebra,
and at least one of $L$ and $A$ be finite-dimensional. Then each
symmetric invariant bilinear form on $L\otimes A$ can be represented
as a sum of decomposable forms $\varphi\otimes \alpha$, $\varphi:
L\times L \to K$, $\alpha: A\times A\to K$ of one of the following $6$
types:
\begin{enumerate}
\item both $\varphi$ and $\alpha$ are cyclic,
\item $\varphi ([L,L],L) = 0$,
\item $\alpha (AA,A) = 0$,
\end{enumerate}
where each of these $3$ types splits into two subtypes: with both
$\varphi$ and $\alpha$ symmetric, and with both $\varphi$ and
$\alpha$ skew-symmetric.
\end{theorem}

\begin{proof}
The proof is absolutely similar to that of Theorem \ref{h2}. As in
the proof of Theorem \ref{h2}, we may write a symmetric invariant
bilinear form $\Phi$ on $L\otimes A$ as $\mathop{\sum}\limits_{i\in
I} \varphi_i \otimes \alpha_i$ for suitable bilinear maps
$\varphi_i:L\times L \to K$ and $\alpha_i:A\times A \to K$. The
invariance condition, written for a given triple $x\otimes a$,
$y\otimes b$, $z\otimes c$, reads:
\begin{equation}\label{symm-bilin}
\mathop{\sum}\limits_{i\in I} \varphi_i([x,y],z) \otimes
\alpha_i(ab,c) +
              \varphi_i([x,z],y) \otimes \alpha_i(ca,b) = 0.
\end{equation}
Symmetrizing this with respect to $x,y$, we get:
\begin{equation*}
\mathop{\sum}\limits_{i\in I} \Big( \varphi_i([x,z],y) +
\varphi_i([y,z],x) \Big) \otimes \alpha_i(ca,b) = 0.
\end{equation*}
Hence the index set can be partitioned $I = I_1 \cup I_2$ in such a
way that
\begin{equation*}
\varphi_i([x,z],y) + \varphi_i([y,z],x) = 0
\end{equation*}
for any $i\in I_1$, and $\alpha_i(AA,A) = 0$ for any $i\in I_2$.
Then (\ref{symm-bilin}) can be rewritten as
\begin{equation*}
\mathop{\sum}\limits_{i\in I_1} \varphi_i([x,y],z) \otimes \Big
(\alpha_i(ab,c) - \alpha_i(ca,b) \Big) = 0.
\end{equation*}
Hence there is a partition $I_1 = I_{11} \cup I_{12}$ such that
$\varphi_i([L,L],L) = 0$ for any $i\in I_{11}$, and $\alpha_i(ab,c)
= \alpha_i(ac,b)$ for any $i\in I_{12}$.

The condition of symmetry of $\Phi$:
\begin{equation*}
\mathop{\sum}\limits_{i\in I} \varphi_i(x,y) \otimes \alpha_i(a,b) -
\varphi_i(y,x) \otimes \alpha_i(b,a) = 0,
\end{equation*}
being symmetrized with respect to $x,y$, allows us to partition
further each of the sets $I_{11},I_{12},I_2$ into two subsets, one
having both $\varphi_i$ and $\alpha_i$ symmetric, and the other one
having both $\varphi_i$ and $\alpha_i$ skew-symmetric.
\end{proof}

This generalizes \cite[Theorem 4.1]{asterisque}, where a similar
statement is proved for $A$ unital.

\section{Derivations}\label{section-der}

Naturally, one may try to apply the same approach to description of
the derivations of a given current algebra $L\otimes A$ (for unital
$A$, see \cite[Corollary 2.2]{low}). Indeed, each derivation $D$ of
$L\otimes A$, being an element of
\begin{equation*}
End (L\otimes A, L \otimes A) \simeq End(L,L) \otimes End(A,A) ,
\end{equation*}
can be expressed in the form $D = \mathop{\sum}\limits_{i\in I}
\varphi_i \otimes \alpha_i$, where $\varphi_i: L \to L$ and
$\alpha_i:A \to A$ are linear maps. The condition that $D$ is a
derivation, written for an arbitrary pair $x\otimes a$ and $y\otimes
b$, where $x,y\in L$ and $a,b\in A$, reads:
\begin{equation*}
\mathop{\sum}\limits_{i\in I} \varphi_i([x,y]) \otimes \alpha_i(ab)
-
              [\varphi_i(x),y] \otimes \alpha_i(a)b -
              [x,\varphi_i(y)] \otimes a\alpha_i(b) = 0.
\end{equation*}
Symmetrizing this equality with respect to $a$, $b$ (this is
equivalent to symmetrization with respect to $x$, $y$):
\begin{equation*}
\mathop{\sum}\limits_{i\in I} \Big( [\varphi_i(x),y] -
[x,\varphi_i(y)] \Big) \otimes
              \Big( a\alpha_i(b) - b\alpha_i(a) \Big) = 0,
\end{equation*}
we get a partition of the index set $I$ into two subsets with
conditions $[\varphi_i(x),y] = [x,\varphi_i(y)]$ and $a\alpha_i(b) =
b\alpha_i(a)$, respectively. But as there are only two variables in
each of $L$ and $A$, no other symmetrization is possible, so the
last equality is all what we can get in this way.

The failure of this method can be also explained by looking at the
simple example of the Lie algebra $sl(2) \otimes tK[t]$. In $sl(2)$,
fix a basis $\{ e_-, h, e_+ \}$ with multiplication
\begin{equation}\label{sl2-basis}
[h,e_-] = - e_-, \quad [h,e_+] = e_+, \quad [e_-,e_+] = h.
\end{equation}
It is easy to check that the map defined, for any $f(t)\in tK[t]$,
by the formulae
\begin{align*}
e_- \otimes f(t) &\mapsto e_- \otimes t\Big(\frac{df(t)}{dt} - f(t)\Big) \\
e_+ \otimes f(t) &\mapsto e_+ \otimes t\Big(\frac{df(t)}{dt} + f(t)\Big) \\
h   \otimes f(t) &\mapsto h \>\>\>\otimes t\frac{df(t)}{dt}
\end{align*}
is a derivation of $sl(2)\otimes tK[t]$. It is
obvious that this map is not 
a decomposable one, i.e., of the form $\varphi \otimes \alpha$ for
some $\varphi: sl(2) \to sl(2)$ and $\alpha: tK[t]\to tK[t]$. But
for this approach to succeed, all the maps in question should be
representable in such a way in the end.

However, under additional assumption on $L\otimes A$, we can derive
information about $Der(L\otimes A)$ from the results of the
preceding sections, using the relationship between $H^1(L,L^*)$,
$H^2(L,K)$, and $\mathcal B(L)$. In the literature, this relationship was
noted many times in a slightly different form, and goes back to the
classical works of Koszul and Hochschild--Serre \cite{hs}. Namely,
there is an exact sequence
\begin{equation}\label{seq}
0 \to H^2(L,K) \overset{u}\to H^1(L,L^*) \overset{v}\to \mathcal B(L)
\overset{w}\to H^3(L,K)
\end{equation}
where for the representative $\varphi\in Z^2(L,K)$ of a given
cohomology class, we have to take the class of $u(\varphi)$, the
latter being given by
\begin{equation*}
(u(\varphi)(x))(y) = \varphi(x,y)
\end{equation*}
for any $x,y\in L$, $v$ is sending the class of a given cocycle
$D\in Z^1(L,L^*)$ to the bilinear form $v(D): L\times L\to K$
defined by the formula
\begin{equation*}
v(D)(x,y) = D(x)(y) + D(y)(x),
\end{equation*}
and $w$ is sending a given symmetric bilinear invariant form
$\psi: L\times L \to K$ to the class of the cocycle $\omega\in
Z^3(L,K)$ defined by
\begin{equation*}
\omega(x,y,z) = \psi([x,y],z)
\end{equation*}
(see, for example, \cite{dzhu}, where a certain long exact sequence
is obtained, of which this one is the beginning, and references
therein for many earlier particular versions; this exact sequence
was also established in \cite[Proposition 7.2]{neeb} with two
additional terms on the right).

In the case where $L \simeq L^*$ as $L$-modules, the sequence (\ref{seq})
provides a way to evaluate $H^1(L,L)$ given $H^2(L,K)$ and $\mathcal B(L)$.
The $L$-module isomorphism $L \simeq L^*$ implies the existence of a
symmetric invariant non-degenerate form $\langle \cdot,
\cdot \rangle$ on $L$. In terms of this form, $u$ is sending the
class of a given cocycle $\varphi\in Z^2(L,K)$ to the class of the
cocycle $u(\varphi)\in Z^1(L,L)$ defined by
\begin{equation*}
\langle (u(\varphi))(x),y \rangle = \varphi(x,y) ,
\end{equation*}
and $v$ is sending the class of a given cocycle $D\in Z^1(L,L)$ to
the bilinear form $v(D): L\times L\to K$ defined by the formula
\begin{equation*}
v(D)(x,y) = \langle D(x),y \rangle + \langle x,D(y) \rangle.
\end{equation*}

Turning to current Lie algebras, we will make even stronger
assumption: that  $L \simeq L^*$ and $A \simeq A^*$. Then, utilizing
the results of preceding sections about $H^2(L\otimes A,K)$ and
$\mathcal B(L\otimes A)$, we will derive results about $Der(L\otimes A)$.

In the literature, given $H^1(L,L^*)$, the space $H^2(L,K)$ was
computed for various Lie algebras $L$ (see, for example,
\cite{santharoubane}, \cite{dzhu} and references therein). Here we
utilize this connection in the other direction.

\begin{theorem}\label{der}
Let $L$ be a nonabelian Lie algebra, $A$  an associative commutative
algebra, both $L$ and $A$ finite-dimensional and with symmetric
invariant non-degenerate bilinear form. Then each derivation of
$L\otimes A$ can be represented as the sum of decomposable linear
maps $d\otimes \beta$, where $d: L \to L$ and $\beta: A \to A$ are
of one of the following types:
\begin{enumerate}
\item
$d([x,y]) = \lambda ([d(x),y] + [x,d(y)])$, $\beta(ab) = \mu
\beta(a)b$ for certain $\lambda, \mu \in K$ such that $\lambda\mu =
1$,
\item
$d([x,y]) = \lambda [d(x),y]$, $\beta(ab) = \mu (\beta(a)b +
a\beta(b))$ for certain $\lambda, \mu \in K$ such that $\lambda\mu =
1$,
\item $[d(x),y] + [x,d(y)] = 0$, $\beta(AA) = 0$, $\beta(a)b = a\beta(b)$,
\item $d([L,L]) = 0$, $[d(x),y] + [x,d(y)] = 0$, $\beta(a)b = a\beta(b)$,
\item $d([L,L]) = 0$, $[d(x),x] = 0$, $\beta(a)b + a\beta(b) = 0$,
\item $[d(x),x] = 0$, $\beta(AA) = 0$, $\beta(a)b + a\beta(b) = 0$,
\item $d([L,L]) = 0$, $d(L) \subseteq \mathcal Z(L)$,
\item $d([L,L]) = 0$, $\beta(A) \subseteq Ann(A)$,
\item $d(L) \subseteq \mathcal Z(L)$, $\beta(AA) = 0$,
\item $\beta(AA) = 0$, $\beta(A) \subseteq Ann(A)$.
\end{enumerate}
\end{theorem}

\begin{proof}
By abuse of notation, let $\langle\cdot\,,\cdot\rangle$ denote a
symmetric invariant non-degenerate bilinear form both on $L$ and
$A$. Obviously, the tensor product of these forms defines a
symmetric invariant non-degenerate bilinear form on $L\otimes A$,
for which by even bigger abuse of notation we will use the same
symbol:
\begin{equation*}
\langle x\otimes a, y\otimes b \rangle = \langle x,y \rangle \langle
a,b \rangle.
\end{equation*}
We have $L^* \simeq L$ as $L$-modules, $A^* \simeq A$ as
$A$-modules, and $(L \otimes A)^* \simeq L \otimes A$ as $L\otimes
A$-modules.

As a vector space, $H^1(L\otimes A, L\otimes A)$ can be represented
as the direct sum of $Ker\,v$ and $Im\,v$, and the exact sequence
(\ref{seq}) tells that $Ker\,v = Im\,u$ and $Im\,v = Ker\,w$.

By Theorem \ref{h2}, $H^2(L\otimes A, K)$ is spanned by cohomology
classes which can be represented by decomposable cocycles $\varphi
\otimes \alpha$ for appropriate $\varphi: L \times L \to L$ and
$\alpha: A \times A \to K$. For each such pair $\varphi$ and
$\alpha$, there are unique linear maps $d: L \to L$ and $\beta: A
\to A$ such that
\begin{equation}\label{d}
\langle d(x),y \rangle = \varphi(x,y)
\end{equation}
for any $x,y\in L$, and
\begin{equation}\label{beta}
\langle \beta(a),b \rangle = \alpha(a,b)
\end{equation}
for any $a,b \in A$. Hence the decomposable linear map $d \otimes
\beta: L\otimes A \to L\otimes A$ satisfies
\begin{equation*}
\langle (d \otimes \beta) (x\otimes a), y\otimes b \rangle =
(\varphi \otimes \alpha) (x\otimes a, y\otimes b) ,
\end{equation*}
i.e., coincides with $u(\varphi \otimes \alpha)$. Thus, $Im\,u$ is
spanned by cohomology classes whose representatives are decomposable
derivations.

Similarly, by the proof of Theorem \ref{forms}, $\mathcal B(L\otimes A)$ is spanned by
decomposable elements $\varphi \otimes \alpha$, and $Ker\,w$ is
spanned by such elements of types (ii) and (iii), i.e., either
$\varphi([L,L],L) = 0$ or $\alpha(AA,A) = 0$. Again, for each such
element we can find $d: L \to L$ and $\beta: A \to A$ satisfying
(\ref{d}) and (\ref{beta}) respectively. Furthermore, we may assume
that for each such $\varphi \otimes \alpha$, the maps $\varphi$ and
$\alpha$ are either both symmetric, or both skew-symmetric, and
hence both $d$ and $\beta$ are either self-adjoint or
skew-self-adjoint, respectively, with respect to $\langle
\cdot,\cdot \rangle$. In both cases we have:
\begin{multline}\label{angles}
\langle (d\otimes \beta)(x\otimes a), y\otimes b \rangle + \langle
x\otimes a, (d\otimes \beta)(y\otimes b) \rangle \\ = \langle d(x),y
\rangle \langle \beta(a),b \rangle + \langle x,d(y) \rangle \langle
a,\beta(b) \rangle = 2 \langle d(x),y \rangle \langle \beta(a),b
\rangle \\ = 2 \varphi(x,y) \otimes \alpha(a,b)
\end{multline}
for any $x,y\in L$ and $a,b\in A$.

The condition $\varphi([L,L],L) = 0$ ensures that $d([L,L]) = 0$,
and an equivalent condition $\varphi(L,[L,L]) = 0$ ensures that
\begin{equation*}
\langle [d(x),z],y \rangle = - \langle d(x), [y,z] \rangle = 0 ,
\end{equation*}
implying $d(L) \subseteq \mathcal Z(L)$, and hence $d\otimes \beta$ is a
derivation of $L\otimes A$. Quite analogously, the condition
$\alpha(AA,A) = \alpha(A,AA) = 0$ implies also that $d\otimes \beta$
is a derivation of $L\otimes A$. The equality (\ref{angles}) ensures
that $v$ maps the cohomology class of this derivation to
$2\varphi\otimes \alpha$. Thus $Im\,v$ is spanned by images of
cohomology classes whose representatives are decomposable
derivations.

Putting all this together, we see that $H^1(L\otimes A, L\otimes A)$
is spanned by the cohomology classes whose representatives are
decomposable derivations. As inner derivations of $L\otimes A$ are,
obviously, also spanned by decomposable inner derivations
$ad\,(x\otimes a) = ad\,x \otimes R_a$, where $R_a$ is the
multiplication on $a\in A$, any derivation of $L\otimes A$ is
representable as the sum of decomposable derivations.

The rest is easy. The condition that $d\otimes \beta$ is a
derivation, reads:
\begin{equation}\label{*}
d([x,y]) \otimes \beta(ab) - [d(x),y] \otimes \beta(a)b - [x,d(y)]
\otimes a\beta(b) = 0
\end{equation}
for any $x,y \in L$, $a,b \in A$. Symmetrizing (\ref{*}) as in the
beginning of this section, we see that either $[d(x),y] = [x,d(y)]$
for any $x,y \in L$, or $\beta(a)b = a\beta(b)$ for any $a,b \in A$.
The equation (\ref{*}) is equivalent to
\begin{equation*}
d([x,y]) \otimes \beta(ab) - [d(x),y] \otimes (\beta(a)b +
a\beta(b)) = 0
\end{equation*}
in the first case, and to
\begin{equation*}
d([x,y]) \otimes \beta(ab) - ([d(x),y] + [x,d(y)]) \otimes \beta(a)b
= 0.
\end{equation*}
in the second case. Now trivial case-by-case considerations
involving vanishing and linear dependence of the linear operators
occurring as tensor product factors in these two equalities, produce
the final list of derivations.
\end{proof}

Theorem \ref{der} can be applied, for example, to the Lie algebra
$\mathfrak g \otimes tK[t]/(t^n)$, $n > 2$, where $\mathfrak g$ is a
semisimple finite-dimensional Lie algebra over any field of
characteristic $0$, to obtain a very short proof of the result of
Benoist \cite{benoist} about realization of any semisimple Lie
algebra as semisimple part of the Lie algebra of derivations of a
nilpotent Lie algebra (another short proof with direct calculation
of $Der(\mathfrak g \otimes tK[t]/(t^3))$ follows from
\cite[Proposition 3.5]{leger-luks}). Indeed, as noted, for example,
in \cite[Lemma 2.2]{bb}, $tK[t]/(t^n)$ possesses a symmetric
nondegenerate invariant bilinear form $B$, hence $\mathfrak g
\otimes tK[t]/(t^n)$ possesses such a form (being the product of the
Killing form on $\mathfrak g$ and $B$), so Theorem \ref{der} is
applicable. As $\mathfrak g$ is perfect and centerless, the derivations of types 
(iv), (v), (vii), (viii), (ix) vanish.
The remaining types can be handled, for example, by appealing to results of 
\cite{hopkins}, \cite{filippov} or \cite{leger-luks}, which imply that 
in the case $\mathfrak g \not\simeq sl(2)$, the corresponding mappings $d$ 
vanish also for types (ii) and (iii), and for the rest of the types
are either inner derivations of $\mathfrak g$, or multiplications by scalar. 
%
%
%
%
%
%
%
%
%
Then, performing elementary calculations with conditions imposed on $\beta$'s in the 
remaining types, and rearranging the obtained spaces of derivations, we get the following 
isomorphism of vector spaces:
\begin{equation}\label{summands}
Der (\mathfrak g \otimes tK[t]/(t^n)) \simeq
(\mathfrak g \otimes K[t]/(t^n)) \oplus (End(\mathfrak g)/\mathfrak g) \oplus
K .
\end{equation}
Elements of the first summand are assembled from types (i) and (x), and act on
$\mathfrak g \otimes tK[t]/(t^n)$ as the Lie multiplication 
by an element of $\mathfrak g$ and the associative commutative multiplication by an
element of $K[t]/(t^n)$.
Elements of the second summand are assembled from types (i), (vi) and (x), and act by the rule
\begin{align*}
&x\otimes t\phantom{^k} \mapsto F(x) \otimes t^{n-1}  \\
&x\otimes t^k           \mapsto 0 \qquad \text{if } k\ge 2 ,
\end{align*}
where $x\in \mathfrak g$, $F \in End(\mathfrak g)$, and elements of $\mathfrak g$ are assumed to be embedded into 
$End(\mathfrak g)$ as inner derivations. 
Elements of the third, one-dimensional, summand are assembled from types (i) and (vi), 
and are proportional to the following derivation:
\begin{align*}
&x\otimes t\phantom{^2} \mapsto \phantom{2}x \otimes t^{n-2}   \\
&x\otimes t^2           \mapsto 2x\otimes t^{n-1}   \\
&x\otimes t^k           \mapsto 0 \qquad \text{if } k \ge 3 ,
\end{align*}
where $x\in \mathfrak g$.

All this implies that, as a Lie algebra, $Der(\mathfrak g \otimes tK[t]/(t^n))$ splits into the semidirect sum
of the semisimple part isomorphic to $\mathfrak g$ (identified with the part 
$\mathfrak g \otimes 1$ of the first summand in (\ref{summands})), and the nilpotent radical consisting of 
$\mathfrak g \otimes tK[t]/(t^n)$ from the first summand, and the whole second and third 
summands.

The case $\mathfrak g = sl(2)$ can be treated separately and easily.

\section{Periodization of semisimple Lie algebras}\label{period}

For a given Lie algebra $L$, its \textit{periodization} is an
$\mathbb N$-graded Lie algebra, with component in each degree
isomorphic to $L$. In other words, the periodization of $L$ is $L
\otimes tK[t]$.

In \cite{larsson}, Anna Larsson studied periodization of semisimple
finite-dimensional Lie algebras $\mathfrak g$ over any field $K$ of
characteristic $0$. She proved that, unless $\mathfrak g$ contains
direct summands isomorphic to $sl(2)$, its periodization possesses a
presentation with only quadratic relations. Since generators and
relations of (generalized graded) nilpotent Lie algebras can be
interpreted as the first and second homology, Larsson's statement
can be formulated in homological terms. 

Note that in \cite{larsson-thesis} this result
was generalized to some classes of Lie superalgebras.
We do not touch upon superalgebras here, and it seems to be an interesting task
to tackle the results and questions from \cite{larsson-thesis} from
this paper's viewpoint.
Interest in periodizations of Lie superalgebras, and whether they admit generators
subject to quadratic relations, arose from an earlier work of
L\"ofwall and Roos \cite{lofwall-roos} about some amazing Hopf algebras (for further
details, see \cite{larsson} and \cite{larsson-thesis}).

As noted in \cite{larsson}, the whole space $H^*(\mathfrak g \otimes
t \mathbb C[t], \mathbb C)$ was studied by much more sophisticated
methods in the celebrated paper by Garland and Lepowsky
\cite{garland-lepowsky} (actually, a particular case interesting for
us here was already sketched in \cite{garland}; the case of
$\mathfrak g = sl(2)$ was also treated in \cite[p.
233]{feigin-fuchs}). Garland and Lepowsky determined the eigenvectors of the
Laplacian on the corresponding Chevalley-Eilenberg (co)chain
complex. However, to extract from \cite{garland-lepowsky} exact
results about (co)homology of interest requires nontrivial
combinatorics with the Weyl group, as demonstrated in
\cite{hanlon-wales}, and case-by-case analysis for each series of
the simple Lie algebras. Here we derive results for the second
cohomology in a uniform way from the results of \S 1 using
elementary methods providing an alternative short proof of
Larsson's result. Also, our approach clearly shows why the case of
$sl(2)$ is exceptional.

\begin{theorem}[Larsson]\label{periodis}
Let $\mathfrak g$ be a finite-dimensional semisimple Lie algebra
over an algebraically closed field $K$ of characteristic $0$. Then
\begin{equation*}
H^2(\mathfrak g \otimes tK[t],K) \simeq C^2(\mathfrak
g,K)/B^2(\mathfrak g,K) \>\bigoplus\> \left(\bigoplus S_\alpha
\right),
\end{equation*}
where the second direct sum is taken over all simple direct summands
of $\mathfrak g$ isomorphic to $sl(2)$, and each $S_\alpha$ is a
certain $5$-dimensional space of symmetric bilinear forms on the
corresponding direct summand. The basic cocycles can be chosen among 
cocycles of the form
\begin{equation}\label{cocycles}
\Phi (x \otimes t^i, y \otimes t^j) = \begin{cases}
\varphi(x,y)  &\text{if } i=j=1,  \\
0             &\text{otherwise},
\end{cases}
\end{equation}
where $x,y \in \mathfrak g$ and $\varphi$ is a skew-symmetric
bilinear form on $\mathfrak g$, and the cocycles $\Psi$ whose only non-vanishing
values on all pairs of the simple direct summands of $\mathfrak g$ are determined by the formula
\begin{equation}\label{cocycles-sl2}
\Psi (x \otimes t^i, y \otimes t^j) = \begin{cases}
\psi(x,y)   &\text{if } i=1, j=2,  \\
- \psi(x,y) &\text{if } i=2, j=1,  \\
0             &\text{otherwise},
\end{cases}
\end{equation}
where $x,y$ belong to the corresponding $sl(2)$-direct summand, and
$\psi$ is a symmetric bilinear form on this direct summand satisfying
\begin{equation}\label{sl2-form}
\psi(e_-,e_+) = \frac 12 \psi(h,h)
\end{equation}
in the standard $sl(2)$-basis {\rm (\ref{sl2-basis})}.
\end{theorem}

\begin{proof}
Consider the cocycles appearing in Theorem \ref{h2} for $L =
\mathfrak g$ and $A = tK[t]$ case-by-case.

\textit{Cocycles of type} (i). Writing the cyclicity condition of
$\alpha$ for the triple $t^{i-1}, t^j, t$, we get
\begin{equation}\label{alpha}
\alpha (t^{i+j-1}, t) = \alpha (t^i, t^j)
\end{equation}
for any $i \ge 2, j \ge 1$.

(ia). $\varphi$ is skew-symmetric, thus $\varphi \in Z^2(\mathfrak
g,K)$, and $\alpha$ is symmetric. Since $H^2(\mathfrak g,K) = 0$, it
follows that $\varphi = d\omega$ for some linear map $\omega:
\mathfrak g\to K$. Define a linear map $\Omega: \mathfrak g \otimes
tK[t] \to K$ by setting, for any $x\in \mathfrak g$,
$$
\Omega (x \otimes t^i) = \begin{cases}
0                            &\text{if }  i = 1, \\
\omega(x) \alpha (t^{i-1},t) &\text{if }  i \ge 2.
\end{cases}
$$
Then, taking (\ref{alpha}) into account, we get $\varphi \otimes
\alpha = d\Omega$, i.e., cocycles of this type are trivial.

(ib). $\varphi$ is symmetric, $\alpha$ is skew-symmetric.

If $\mathfrak g = sl(2)$, direct calculation shows that the space of
corresponding $\varphi$'s coincides with the space of all symmetric
bilinear forms on $sl(2)$ satisfying the condition (\ref{sl2-form}),
and hence is $5$-dimensional (an equivalent calculation is contained
in \cite[Theorem 6.5]{alia-d}). This case is exceptional, as shows
the following

\begin{lemma}[Dzhumadil'daev--Bakirova]\label{db}
Let $\mathfrak g\not\simeq sl(2)$ be simple. Then any symmetric
bilinear form $\varphi$ on $\mathfrak g$ satisfying
\begin{equation}\label{cocycle}
\varphi([x,y],z) + \varphi([z,x],y) + \varphi([y,z],x) = 0
\end{equation}
for any $x,y,z\in \mathfrak g$ vanishes.
\end{lemma}

This Lemma is proved in \cite{alia-db} by considering the Chevalley
basis of $\mathfrak g$ and performing computations with the
corresponding root system. In \cite{alia-d} and \cite{alia-db},
symmetric bilinear forms satisfying the condition (\ref{cocycle})
are called  \textit{commutative $2$-cocycles} and arise naturally in
connection with classification of algebras satisfying skew-symmetric
identities. We will give a different proof which stresses the
connection with yet other notions and results.

\begin{proof} Consider a map from the space of
bilinear forms on a Lie algebra $L$ to the space of linear maps from
$L$ to $L^*$, by sending a bilinear form $\varphi$ to the linear map
$D: L \to L^*$ such that ${D(x)(y) = \varphi(x,y)}$. It is easy to see
that a symmetric bilinear form $\varphi$ satisfying (\ref{cocycle})
maps to a linear map $D$ satisfying
$$
D([x,y]) = - y \bullet D(x) + x \bullet D(y) ,
$$
where $\bullet$ denotes the standard $L$-action on the dual module
$L^*$ (this is completely analogous to the embedding of $H^2(L,K)$
into $H^1(L,L^*)$ mentioned in \S \ref{section-der}).

For any finite dimensional simple Lie algebra, there is an
isomorphism of $\mathfrak g$-modules $\mathfrak g \simeq \mathfrak
g^*$, and we have an embedding of the space of bilinear forms in
question into the space of linear maps $D:\mathfrak g \to \mathfrak
g$ satisfying the condition
$$
D([x,y]) = -[D(x),y] - [x,D(y)] .
$$
Such maps are called \textit{antiderivations} and were studied in
several papers. In particular, in \cite[Theorem 5.1]{hopkins} it is
proved that central simple classical Lie algebras of dimension $>3$
do not have nonzero antiderivations (generalizations of this result
in different directions obtained further in the series of papers by
Filippov, of which \cite{filippov} is the latest, and in
\cite{leger-luks}). Hence, these Lie algebras do not have nonzero
symmetric bilinear forms satisfying (\ref{cocycle}) either.
\end{proof}

Coupled with the well-known fact that $H^2(\mathfrak g,K)=0$, Lemma \ref{db} 
implies that \textit{any} (symmetric, skew-symmetric or mixed) bilinear form 
$\varphi$ on $\mathfrak g \not\simeq sl(2)$ satisfying the cocycle equation 
(\ref{cocycle}), is a $2$-coboundary\footnote{
Added December 26, 2017: The statement is true, but it does not follow, at 
least in a straightforward way, from Lemma \ref{db} and the fact that 
$H^2(\mathfrak g,K)=0$. For a correct proof, see a forthcoming text (with 
A. Makhlouf) ``Hom-Lie algebra structures on Kac-Moody algebras''.
}.
This can be also compared with the fact that Leibniz cohomology
(and, in particular, the second Leibniz cohomology  with trivial
coefficients) of $\mathfrak g$ vanishes (for homological version,
see \cite[Proposition 2.1]{pirashvili} or \cite{ntolo}). The
condition for a bilinear form $\varphi$ to be a Leibniz
$2$-cocycle can be expressed as
\begin{equation*}
\varphi([x,y],z) + \varphi([z,x],y) - \varphi(x,[y,z]) = 0 .
\end{equation*}

In the general case, where $\mathfrak g$ is a direct sum of simple
ideals $\mathfrak g_1 \oplus \dots \oplus \mathfrak g_n$, it is easy
to see that $\varphi(\mathfrak g_i, \mathfrak g_j) = 0$ for $i\ne
j$, and hence the space of symmetric bilinear forms on $\mathfrak g$
satisfying ($\ref{cocycle}$) decomposes into the direct sum of
appropriate spaces on each of $\mathfrak g_i$. The latter are
determined by (\ref{sl2-form}) if $\mathfrak g_i \simeq sl(2)$ and
vanishes otherwise.

Now, turning to $\alpha$'s, and permuting $i$ and $j$ in (\ref{alpha}), 
we get $\alpha (t^i, t^j) = 0$ for all $i,j \ge 2$ and
$\alpha(t^k, t) = 0$ for all $k\ge 3$. Conversely, it is easy to see
that a skew-symmetric map $\alpha$ satisfying these conditions is
cyclic. Hence, the space of skew-symmetric cyclic maps on $tK[t]$ is
1-dimensional and each cocycle of this type can be written in the
form (\ref{cocycles-sl2}).

\medskip\textit{Cocycles of type} (ii).

(iia). $\varphi$ is skew-symmetric, $\alpha$ is symmetric.

\begin{lemma}\label{skew-symm-cyclic}
Any skew-symmetric cyclic bilinear form on $\mathfrak g$ vanishes.
\end{lemma}

\begin{proof}
The proof is almost identical to the proof of the well-known fact
that any skew-symmetric invariant bilinear form on $\mathfrak g$ vanishes
(see, for example, \cite[Chapter 1, \S 6, Exercises 7(b) and
18(a,b)]{bourbaki}). Namely, let $\varphi$ be a skew-symmetric
cyclic form on $\mathfrak g$. First, let $\mathfrak g$ be simple.
There is a linear map $\sigma: \mathfrak g \to \mathfrak g$ such
that $\varphi(x,y) = \langle \sigma(x),y \rangle$, where
$\langle\cdot,\cdot\rangle$ is the Killing form on $\mathfrak g$.
Then, for any $x,y,z\in \mathfrak g$:
$$
\langle\sigma([x,y]),z\rangle = \varphi([x,y],z) = -
\varphi(y,[z,x]) = -\langle\sigma(y),[z,x]\rangle =
-\langle[x,\sigma(y)],z\rangle.
$$
Hence $\sigma$ anticommutes with each $ad\,x$ and, in particular,
$[\sigma(x),x] = 0$ for any $x\in \mathfrak g$. But then by
\cite[Lemme 2]{benoist} (or by more general results from
\cite{leger-luks}), $\sigma$ belongs to the centroid of $\mathfrak
g$. Hence $\sigma$ is a scalar and necessarily vanishes.

In the general case where $\mathfrak g$ is the direct sum of simple
ideals $\mathfrak g_1 \oplus \dots \oplus \mathfrak g_n$, it is easy
to see that $\varphi(\mathfrak g_i,\mathfrak g_j) = 0$ for any $i\ne
j$. But by just proved $\varphi$ vanishes also on each $\mathfrak
g_i$, and hence vanishes on the whole of $\mathfrak g$.
\end{proof}

(iib). $\varphi$ is symmetric and $\alpha$ is skew-symmetric, so
$\alpha \in HC^1(tK[t])$. There is an isomorphism of graded algebras
\begin{equation*}
HC^*(K[t]) \simeq HC^*(K) \oplus HC^*(tK[t])
\end{equation*}
(for a general relationship between cyclic homology of augmented
algebra and its augmentation ideal, see \cite[\S 4]{loday-quillen};
the cohomological version can be obtained in exactly the same way).
On the other hand,
\begin{equation*}
HC^*(K[t]) \simeq HC^*(K) \oplus \text{(terms concentrated in degree
$0$)}
\end{equation*}
(see, for example, \cite[\S 3.1.7]{loday}). Hence $HC^1(tK[t]) = 0$.
Of course, the vanishing of $HC^1(tK[t])$ can be established also by
direct easy calculations.

\textit{Cocycles of type} (iii) vanish.

\textit{Cocycles of type} (iv). Obviously, $\alpha(t^i,t^j) = 0$ for
$(i,j) \ne (1,1)$. Hence $\alpha$ is symmetric, and each cocycle of
this type has the form (\ref{cocycles}).

To summarize: cocycles of type (ia), (ii) and (iii) either are
trivial or vanish, and cocycles of type (ib) are given by formula
(\ref{cocycles-sl2}) and vanish if $\mathfrak g$ does not contain
direct summands isomorphic to $sl(2)$. So in the latter case, all
nontrivial cocycles are of type (iv) and given by formula
(\ref{cocycles}). Considering the natural grading in $\mathfrak g
\otimes tK[t]$ by degrees of $t$, observing that the second
cohomology is finite-dimensional and hence is dual to the second
homology, and turning to interpretation of the $2$-cycles as
relations between generators, we get the assertion proved in
\cite{larsson}: that $\mathfrak g \otimes tK[t]$ admits a
presentation with quadratic relations, provided $\mathfrak g$ does
not contain direct summands isomorphic to $sl(2)$.

Let us decide now when cocycles given by (\ref{cocycles}) and
(\ref{cocycles-sl2}) are cohomologically independent. According to
Remark \ref{indep} in \S \ref{cohomology}, any cohomological
dependency beyond linear dependency can occur for cocycles of type
(\ref{cocycles}) only. Writing, as in Remark \ref{indep}, $\Phi =
d\Omega$ for $\Omega = \mathop{\sum}\limits_{i\in I} \omega_i
\otimes \beta_i$, where $\omega_i: \mathfrak g \to \mathfrak g$ and
$\beta_i: tK[t] \to tK[t]$ are some linear maps, we get:
\begin{equation*}
\mathop{\sum}\limits_{i\in I} \omega_i([x,y]) \beta_i(t^2) =
\varphi(x,y)
\end{equation*}
and
\begin{equation*}
\mathop{\sum}\limits_{i\in I} \omega_i([x,y]) \beta_i(t^k) = 0
\end{equation*}
for $k\ge 3$. Hence it is clear that cocycles of type
(\ref{cocycles}) are cohomologically independent if and only if the
corresponding skew-symmetric bilinear forms are independent modulo
$2$-coboundaries.

\end{proof}

In \cite{larsson}, the author speculates about the possibility to
derive Theorem \ref{periodis} from the standard presentation of the
affine Kac-Moody algebra. Let us indicate briefly how one can do the
opposite: namely, how Theorem \ref{periodis} allows us to recover the
Serre relations between Chevalley generators of non-twisted affine
Kac-Moody algebras.

Let $\mathfrak g$ be a finite-dimensional simple Lie algebra over an
algebraically closed field $K$ of characteristic $0$, and $\widehat
L = \mathfrak g \otimes K[t,t^{-1}]$ be a loop algebra, aka ``centerless,
derivation-free'' part of an affine non-twisted Kac-Moody algebra. It
is well-known that it admits a triangular decomposition which, with
slight rearrangements of terms, can be written in the form
\begin{equation*}
\widehat L = \Big( (\mathfrak g\otimes t^{-1} K[t^{-1}]) \oplus
(\mathfrak n_- \otimes 1) \Big) \oplus \Big( \mathfrak h \otimes 1
\Big) \oplus \Big( (\mathfrak g\otimes tK[t]) \oplus (\mathfrak n_+
\otimes 1) \Big),
\end{equation*}
where $\mathfrak g = \mathfrak n_- \oplus \mathfrak h \oplus
\mathfrak n_+$ is the triangular decomposition of the simple Lie
algebra $\mathfrak g$ (\cite[\S 7.6]{kac}; all direct sums are
direct sums of vector spaces).

Consider the Hochschild--Serre spectral sequence abutting to
$H^*((\mathfrak g\otimes tK[t]) \oplus (\mathfrak n_+ \otimes 1),
K)$ with respect to the ideal $\mathfrak g\otimes tK[t]$. The
general result implicitly contained in \cite{hs} and explicitly, for
example, in \cite[Lemma 1]{rozen}, tells that this spectral sequence
degenerates at the $E_2$ term. For the sake of simplicity, let us
exclude the case where $\mathfrak g = sl(2)$. Then the $E_2$ terms
affecting the second cohomology in question are:
\begin{align*}
E_\infty^{02} = E_2^{02} &= H^2 (\mathfrak n_+, H^0 (\mathfrak
g\otimes tK[t], K)) \simeq
                            H^2 (\mathfrak n_+, K),            \\
E_\infty^{11} = E_2^{11} &= H^1 (\mathfrak n_+, H^1 (\mathfrak
g\otimes tK[t], K)) \simeq
                            H^1 (\mathfrak n_+, \mathfrak g),  \\
E_\infty^{20} = E_2^{20} &= H^0 (\mathfrak n_+, H^2 (\mathfrak
g\otimes tK[t], K)) \simeq
                            (C^2(\mathfrak g, K) / B^2(\mathfrak g, K))^{\mathfrak n_+}.
\end{align*}
The first and second isomorphisms here are obvious, the third one
follows from Theorem \ref{periodis}. Here the first term corresponds
to relations between elements of $\mathfrak n_+ \otimes 1$, which
are (classical) Serre relations for the finite-dimensional Lie
algebra $\mathfrak g$, the second term corresponds to relations
between elements of $\mathfrak g\otimes tK[t]$ and $\mathfrak n_+
\otimes 1$, and the third one corresponds to relations between
elements of $\mathfrak g\otimes tK[t]$. Expressing these elements by means of 
Chevalley generators in terms of the corresponding Cartan matrix, we get the
corresponding part of the Serre relations.

Repeating similar reasonings for the ``minus'' part, and
completing relations in an obvious manner between the ``plus'' and
``minus'' parts and the ``Cartan subalgebra'' $\mathfrak h \otimes
1$, we get the complete set of defining relations for $\widehat L$.
The whole affine non-twisted Kac-Moody algebra is obtained from
$\widehat L$ by the well-known construction which adds central
extension and derivation, and its presentation readily follows from
the presentation of $\widehat L$.

This approach is by no means new. For example, we find in
\cite[Remark in \S 2]{leites}: ``Similar calculations by induction
on the rank for simple finite-dimensional and loop algebras give the
shortest known to us proof of completeness of the Serre defining
relations''. ``Induction on the rank'' means induction with
repetitive application of the Hochschild--Serre spectral sequence
relative to a Kac-Moody algebra build upon the simple
finite-dimensional subalgebra of $\mathfrak g$. Here we use the
Hochschild-Serre spectral sequence in a different way, and only
once, thus getting even shorter proof.

It is mentioned in \cite[\S 9.16]{kac} that ``a simple cohomological
proof'' of the completeness of the Serre defining relations ``was
found by O. Mathieu (unpublished)''. We presume that the approach
outlined here is similar to that unpublished proof.

\section*{Acknowledgements}

Thanks are due to Semyon Konstein, Dimitry Leites, and Irina Shchepochkina for a 
careful reading of preliminary versions
of the manuscript and many suggestions which improved the presentation,
to Askar Dzhumadil'daev for drawing attention to papers \cite{alia-d}
and \cite{alia-db} which provide an interesting connection between
the topic of this paper and classification of skew-symmetric
identities of algebras, 
to Clas L\"ofwall for drawing attention to the thesis \cite{larsson-thesis},
to Anna Tihinen (Larsson) for supplying it, 
and to Rutwig Campoamor-Stursberg for supplying some other needed literature.


\begin{thebibliography}{NW}

\bibitem[BB]{bb} I. Bajo and S. Benayadi,
\emph{Lie algebras admitting a unique quadratic structure}, Comm.
Algebra \textbf{25} (1997), 2795--2805.

\bibitem[Be]{benoist} Y. Benoist,
\emph{La partie semi-simple de l'alg\`ebre des d\'erivations d'une
alg\`ebre de Lie nilpotente}, Comptes Rendus Acad. Sci. Paris
\textbf{307} (1988), 901--904.

\bibitem[BK]{berezin} F. Berezin and F.I. Karpelevich, 
\emph{Lie algebras with supplementary structure},
Math. USSR Sbornik \textbf{6} (1968), 185--203.

\bibitem[Bo]{bourbaki} N. Bourbaki, 
\emph{Lie Groups and Lie Algebras}, Chapters 1--3, Springer, 1989.

\bibitem[D1]{dzhu} A.S. Dzhumadil'daev, \emph{Symmetric (co)homologies of Lie algebras},
Comptes Rendus Acad. Sci. Paris \textbf{324} (1997), 497--502.

\bibitem[D2]{alia-d} \bysame, \emph{Algebras with skew-symmetric identity of degree $3$},
J. Math. Sci. \textbf{161} (2009), 11--30.

\bibitem[DB]{alia-db} \bysame{} and A.B. Bakirova, 
\emph{Simple two-sided anti-Lie-admissible algebras}, 
J. Math. Sci. \textbf{161} (2009), 31--36.

\bibitem[FF]{feigin-fuchs} B.L. Feigin and D.B. Fuchs,
\emph{Cohomology of some nilpotent subalgebras of the Virasoro and
Kac-Moody algebras}, J. Geom. Phys. \textbf{5} (1988), 209--235.

\bibitem[F]{filippov} V.T. Filippov, \emph{$\delta$-derivations of prime Lie algebras},
Siber. Math. J. \textbf{40} (1999), 174--184.

\bibitem[G]{garland} H. Garland,
\emph{Dedekind's $\eta$-function and the cohomology of infinite
dimensional Lie algebras}, Proc. Nat. Acad. Sci. USA \textbf{72}
(1975), 2493--2495.

\bibitem[GL]{garland-lepowsky} \bysame{ }and J. Lepowsky,
\emph{Lie algebra homology and the Macdonald-Kac formulas}, Invent.
Math. \textbf{34} (1976), 37--76.

\bibitem[Ha]{haddi} A. Haddi,
\emph{Homologie des alg\`ebres de Lie \'etendues \`a une alg\`ebre
commutative}, Comm. Algebra \textbf{20} (1992), 1145--1166.

\bibitem[HW]{hanlon-wales} P. Hanlon and D.B. Wales,
\emph{On the homology of $sl_n(t \mathbb C[t])$ and a theorem of
Stembridge}, J. Algebra \textbf{269} (2003), 1--17.

\bibitem[HS]{hs} G. Hochschild and J.-P. Serre, \emph{Cohomology of Lie algebras},
Ann. Math. \textbf{57} (1953), 591--603.

\bibitem[Ho]{hopkins} N.C. Hopkins,
\emph{Generalized derivations of nonassociative algebras}, Nova J.
Math. Game Theory Algebra \textbf{5} (1996), 215--224.

\bibitem[K]{kac} V. Kac, \emph{Infinite Dimensional Lie Algebras}, 3rd ed.,
Cambridge Univ. Press, 1995.

\bibitem[La1]{larsson} A. Larsson, \emph{A periodisation of semisimple Lie algebras},
Homology, Homotopy and Appl. \textbf{4} (2002), 337--355.

\bibitem[La2]{larsson-thesis} \bysame, 
\emph{Periodisations of Contragredient Lie Superalgebras and Their Presentations},
PhD thesis, Stockholm University, 2003.

\bibitem[LL]{leger-luks} G.F. Leger and E.M. Luks,
\emph{Generalized derivations of Lie algebras}, J. Algebra
\textbf{228} (2000), 165--203.

\bibitem[LP]{leites} D. Leites and E. Poletaeva,
\emph{Defining relations for classical Lie algebras of polynomial
vector fields}, Math. Scand. \textbf{81} (1997), 5--19;
\textsf{arXiv:math/0510019}.

\bibitem[Lo]{loday} J.-L. Loday, \emph{Cyclic Homology}, 2nd ed., Springer, 1998.

\bibitem[LQ]{loday-quillen} \bysame{ }and D. Quillen,
\emph{Cyclic homology and the Lie algebra homology of matrices},
Comment. Math. Helv. \textbf{59} (1984), 565--591.

\bibitem[LR]{lofwall-roos} C. L\"ofwall and  J.-E. Roos,
\emph{A nonnilpotent 1-2-presented graded Hopf algebra whose Hilbert series converges
in the unit circle}, Adv. Math. \textbf{130} (1997), 161--200.

\bibitem[NW]{neeb} K.-H. Neeb and F. Wagemann,
\emph{The second cohomology of current algebras of general Lie
algebras}, Canad. J. Math. \textbf{60} (2008), 892--922;
\textsf{arXiv:math/0511260}.

\bibitem[N]{ntolo} P. Ntolo,
\emph{Homologie de Leibniz d'alg\`ebres de Lie semi-simples},
Comptes Rendus Acad. Sci. Paris \textbf{318} (1994), 707--710.

\bibitem[P]{pirashvili} T. Pirashvili, \emph{On Leibniz homology},
Ann. Inst. Fourier \textbf{44} (1994), 401--411.

\bibitem[R]{rozen} B.I. Rozenfeld,
\emph{Cohomology of the algebra of formal universal differential
operators}, Funct. Anal. Appl. \textbf{9} (1975), 126--130.

\bibitem[S]{santharoubane} L.J. Santharoubane,
\emph{The second cohomology group for Kac-Moody Lie algebras and
K\"ahler differentials}, J. Algebra \textbf{125} (1989), 13--26.

\bibitem[Z1]{asterisque} P. Zusmanovich,
\emph{The second homology group of current Lie algebras},
Ast\'erisque \textbf{226} (1994), 435--452;
\textsf{arXiv:0808.0217}.

\bibitem[Z2]{low} \bysame,
\emph{Low-dimensional cohomology of current Lie algebras and analogs
of the Riemann tensor for loop manifolds}, Lin. Algebra Appl.
\textbf{407} (2005), 71--104; \textsf{arXiv:math/0302334}.

\end{thebibliography}
\end{document}